\documentclass[12pt,reqno]{article}

\usepackage[usenames]{color}
\usepackage{amssymb}
\usepackage{graphicx}
\usepackage{amscd}

\usepackage{amsthm}
\newtheorem{theorem}{Theorem}

\newtheorem{proposition}[theorem]{Proposition}
\newtheorem{corollary}[theorem]{Corollary}

\theoremstyle{definition}
\newtheorem{definition}[theorem]{Definition}
\newtheorem{example}[theorem]{Example}

\usepackage[colorlinks=true,
linkcolor=webgreen, filecolor=webbrown,
citecolor=webgreen]{hyperref}

\definecolor{webgreen}{rgb}{0,.5,0}
\definecolor{webbrown}{rgb}{.6,0,0}

\usepackage{color}

\usepackage{float}

\usepackage{graphics,amsmath,amssymb}
\usepackage{amsfonts}
\usepackage{latexsym}
\usepackage{epsf}

\newcommand{\seqnum}[1]{\href{http://oeis.org/#1}{\underline{#1}}}

\begin{document}

\begin{center}
\vskip 1cm{\LARGE\bf Characterizations of the Borel triangle and Borel polynomials} \vskip 1cm \large
Paul Barry\\
School of Science\\
Waterford Institute of Technology\\
Ireland\\
\href{mailto:pbarry@wit.ie}{\tt pbarry@wit.ie}
\end{center}
\vskip .2 in

\begin{abstract} We use Riordan array theory to give characterizations of the Borel triangle and its associated polynomial sequence. We show that the Borel polynomials are the moment sequence for a family of orthogonal polynomials whose coefficient array is a Riordan array. The role of the Catalan matrix in defining the Borel triangle is examined. We generalize the Borel triangle to a family of two parameter triangles. Generating functions are expressed as Jacobi continued fractions, as well as the zeros of appropriate quadratic expressions. The Borel triangle is exhibited as a Hadamard product of matrices. We investigate the reversions of the triangles studied. We introduce the notion of Fuss-Borel triangles and Fuss-Catalan triangles. We end with some remarks on the Catalan triangle. \end{abstract}

\section{Introduction} The Borel triangle \seqnum{A234950}, which begins
$$\left(
\begin{array}{ccccccc}
 1 & 0 & 0 & 0 & 0 & 0 & 0 \\
 2 & 1 & 0 & 0 & 0 & 0 & 0 \\
 5 & 6 & 2 & 0 & 0 & 0 & 0 \\
 14 & 28 & 20 & 5 & 0 & 0 & 0 \\
 42 & 120 & 135 & 70 & 14 & 0 & 0 \\
 132 & 495 & 770 & 616 & 252 & 42 & 0 \\
 429 & 2002 & 4004 & 4368 & 2730 & 924 & 132 \\
\end{array}
\right),$$ has recently appeared in a number of contexts. As detailed in the article \cite{Counting}, these include  pseudo-triangulations of point
sets \cite{Tri} (see also \cite{Arc}), the Betti numbers of certain principal Borel ideals \cite{Borel}, Cambrian Hopf
algebras \cite{Cambrian}, quantum physics \cite{Quantum}, and permutation patterns \cite{Perm}. In addition they appear in the theory of parking functions on trees \cite{Parking}, and in relation to Riordan arrays \cite{GenPas}. A signed version of this triangle is \seqnum{A062991} \cite{Signed}.

Closely related to this triangle is the Catalan triangle \seqnum{A009766} which begins
$$\left(
\begin{array}{ccccccc}
 1 & 0 & 0 & 0 & 0 & 0 & 0 \\
 1 & 1 & 0 & 0 & 0 & 0 & 0 \\
 1 & 2 & 2 & 0 & 0 & 0 & 0 \\
 1 & 3 & 5 & 5 & 0 & 0 & 0 \\
 1 & 4 & 9 & 14 & 14 & 0 & 0 \\
 1 & 5 & 14 & 28 & 42 & 42 & 0 \\
 1 & 6 & 20 & 48 & 90 & 132 & 132 \\
\end{array}
\right).$$
The elements of this matrix are also called ballot numbers. In this note these will be denoted by $C_{n,k}$ where $$C_{n,k}=\frac{n-k+1}{k+1}\binom{n+k}{k}.$$

In this paper, we use the theory of Riordan arrays to help find characterizations for this triangle. Thus we start with an overview of Riordan arrays.

\section{Preliminaries on Riordan arrays}
We recall some facts about Riordan arrays in this introductory section. Readers familiar with Riordan arrays may wish to move on to the next section. A Riordan array \cite{Book, SGWW} is defined by a pair of power series
$$g(x)=g_0 + g_1 x + g_2 x^2 + \cdots=\sum_{n=0}^{\infty} g_n x^n,$$ and
$$f(x)=f_1 x + f_2 x^2+ f_3 x^3 + \cdots = \sum_{n=1}^{\infty} f_n x^n.$$
We require that $g_0 \ne 0$ (which ensures that $g(x)$ is invertible, with inverse $\frac{1}{g(x)}$), while we also demand that
$f_0=0$ and $f_1 \ne 0$ (hence $f(x)$ has a compositional inverse $\bar{f}(x)=\text{Rev}(f)(x)$ defined by $f(\bar{f}(x))=x$).
The set of such pairs $(g(x), f(x))$ forms a group (called the Riordan group \cite{SGWW}) with multiplication
$$(g(x), f(x)) \cdot (u(x), v(x))=(g(x)u(f(x)), v(f(x)),$$ and with inverses given by
$$(g(x), f(x))^{-1}=\left(\frac{1}{g(\bar{f}(x))}, \bar{f}(x)\right).$$
The coefficients of the power series may be drawn from any ring (for example, the integers $\mathbb{Z}$) where these operations make sense. To each such ring there exists a corresponding Riordan group.

There is a matrix representation of this group, where to the element $(g(x), f(x))$ we associate the matrix
$\left(a_{n,k}\right)_{0 \le n,k \le \infty}$ with general element
$$a_{n,k}=[x^n] g(x)f(x)^k.$$
Here, $[x^n]$ is the functional that extracts the coefficient of $x^n$ in a power series. In this representation, the group law corresponds to ordinary matrix multiplication, and the inverse of $(g(x), f(x))$ is represented by the inverse of $\left(a_{n,k}\right)_{0 \le n,k \le \infty}$.

The Fundamental Theorem of Riordan arrays is the rule
$$(g(x), f(x))\cdot h(x)=g(x)h(f(x)),$$ detailing how an array $(g(x), f(x))$ can act on a power series. This corresponds to the matrix $(a_{n,k})$ multiplying the vector $(h_0, h_1, h_2,\ldots)^T$.

\begin{example} Pascal's triangle, also known as the binomial matrix, is defined by the Riordan group element
$$\left(\frac{1}{1-x}, \frac{x}{1-x}\right).$$ This means that we have
$$\binom{n}{k}=[x^n] \frac{1}{1-x} \left(\frac{x}{1-x}\right)^k.$$
To see that this is so, we need to be familiar with the rules of operation of the functional $[x^n]$ \cite{MC}.
We have
\begin{align*}
[x^n] \frac{1}{1-x} \left(\frac{x}{1-x}\right)^k&=[x^n] \frac{x^k}{(1-x)^{k+1}}\\
&= [x^{n-k}] (1-x)^{-(k+1)}\\
&= [x^{n-k}] \sum_{j=0}^{\infty} \binom{-(k+1)}{j}(-1)^j x^j\\
&= [x^{n-k}] \sum_{j=0}^{\infty} \binom{k+1+j-1}{j}x^j\\
&= [x^{n-k}] \sum_{j=0}^{\infty} \binom{k+j}{j} x^j \\
&= \binom{k+n-k}{n-k}=\binom{n}{n-k}=\binom{n}{k}.\end{align*}
\end{example}

The \emph{derivative subgroup} of the Riordan group is the set of elements of the form $(f'(x), f(x))$. That this set is a subgroup follows from the rules for multiplication in the group and the chain rule for differentiation. If $f(x)=\frac{x}{1+rx+sx^2}$, then the corresponding element of this subgroup is given by
$$\left(\left(\frac{x}{1+rx+sx^2}\right)',\frac{x}{1+rx+sx^2}\right)=\left(\frac{1-sx^2}{(1+rx+sx^2)^2}, \frac{x}{1+rx+sx^2}\right).$$

The \emph{production matrix} of a Riordan array $M$ \cite{PM_1, PM_2} is the matrix $M^{-1}\cdot \overline{M}$, where $\overline{M}$ is the matrix $M$ with its top row removed. Production matrices are Hessenberg matrices, and in the case of Riordan arrays, the columns after the first are all shifted versions of the second column. Riordan arrays that are the coefficient arrays of orthogonal polynomials have production matrices that are tri-diagonal \cite{Classical}.

Note that all the arrays in this note are lower triangular matrices of infinite extent. We show appropriate truncations.

Many examples of sequences and  Riordan arrays are documented in the On-Line Encyclopedia of Integer Sequences (OEIS) \cite{SL1, SL2}. Sequences are frequently referred to by their
OEIS number. For instance, the binomial matrix $\left(\frac{1}{1-x}, \frac{x}{1-x}\right)$ (``Pascal's triangle'') is \seqnum{A007318}. In the sequel we will not distinguish between an array pair $(g(x), f(x))$ and its matrix representation.

An exponential Riordan array is defined by a pair of power series $g(x)=\sum_{n=0}^{\infty}g_n \frac{x^n}{n!}$ and $f(x)=\sum_{n=1}^{\infty}f_n \frac{x^n}{n!}$, (with $g_0 \ne 0$, and $f_1 \ne 0$) where the  matrix representation is given by
$$a_{n,k}=\frac{n!}{k!} [x^n] g(x)f(x)^k.$$ 
By the \emph{reversal} of a number triangle $a_{n,k}$, we mean the triangle with $(n,k)$-th term given by $a_{n,n-k}$.
The Iverson notation $[\mathcal{P}]$ \cite{Concrete} evaluates to $1$ if the proposition $\mathcal{P}$ is true, and to $0$ otherwise.

We shall have occasion to use Lagrange inversion in this note \cite{LI}. The version that we shall use is the following.  We have
$$[x^n] G(\bar{f})= \frac{1}{n}[x^{n-1}]G'(x)\left(\frac{x}{f}\right)^n.$$

\section{Defining the Borel triangle}
There are many approaches possible to defining the Borel triangle. The approach that we take will allow us to define a family of generalized Borel triangles. The methods we use will be partly based on the theory of Riordan arrays.

We begin by considering the Riordan array
$$\left(\frac{1}{1+x}, \frac{x(1-yx)}{(1+x)^2}\right).$$ By the theory of Riordan arrays, the inverse of this matrix is given by
$$\left(1+\text{Rev}\left(\frac{x(1-yx)}{(1+x)^2}\right),\text{Rev}\left(\frac{x(1-yx)}{(1+x)^2}\right)\right).$$
This matrix begins
$$\left(
\begin{array}{ccccc}
 1 & 0 & 0 & 0 & 0 \\
 1 & 1 & 0 & 0 & 0 \\
 y+2 & y+3 & 1 & 0 & 0 \\
 2 y^2+6 y+5 & 2 y^2+8 y+9 & 2 y+5 & 1 & 0 \\
 5 y^3+20 y^2+28 y+14 & 5 y^3+25 y^2+44 y+28 & 5 y^2+19 y+20 & 3 y+7 & 1 \\
\end{array}
\right).$$
This indicates that the first column, less its first element, is a polynomial sequence in $y$ whose coefficient array coincides with the start of the Borel triangle. We take this as a definition.
\begin{definition} The Borel triangle is the triangle $B_{n,k}$ defined by
$$B_{n,k}=[x^n y^k] \frac{1}{x} \text{Rev}\left(\frac{x(1-yx)}{(1+x)^2}\right).$$
\end{definition}
We then have the following expression for $B_{n,k}$.
\begin{proposition}
We have
$$B_{n,k}=\frac{1}{n+1} \binom{2n+2}{n-k}\binom{n+k}{k}.$$
\end{proposition}
\begin{proof}
We have
$$[x^n] \frac{1}{x}\text{Rev}\left(\frac{x(1-yx)}{(1+x)^2}\right)=[x^{n+1}] \text{Rev}\left(\frac{x(1-yx)}{(1+x)^2}\right)$$ by the properties of the functional $[x^n]$ \cite{MC}.
We now use Lagrange inversion \cite{LI} to simplify this. We have
\begin{align*}
[x^{n+1}] \text{Rev}\left(\frac{x(1-yx)}{(1+x)^2}\right)&=\frac{1}{n+1}[x^n] \left(\frac{(1+x)^2}{1-yx}\right)^{n+1}\\
&=\frac{1}{n+1}[x^n] \sum_{j=0}^{2n+2}\binom{2n+2}{j}x^j \sum_{i=0}^{\infty} \binom{-(n+1)}{i}(-1)^ix^iy^i\\
&=\frac{1}{n+1}[x^n]\sum_{j=0}^{2n+2}\binom{2n+2}{j}x^j \sum_i \binom{n+i}{i} x^i y^i\\
&=\frac{1}{n+1} \sum_{j=0}^{2n+2} \binom{2n+2}{j} \binom{2n-j}{n-j}y^{n-j}.\end{align*}
The result now follows since
$$[y^k] \frac{1}{n+1} \sum_{j=0}^{2n+2} \binom{2n+2}{j} \binom{2n-j}{n-j}y^{n-j}=\frac{1}{n+1}\binom{2n+2}{n-k}\binom{n+k}{k}.$$
\end{proof}
The generating function of $B_{n,k}$ is then given as follows.
\begin{proposition}
The generating function $B(x,y)$ of the triangle $B_{n,k}$ is given by
$$B(x,y)=\frac{1-2x-\sqrt{1-4x(y+1)}}{2x(x+y)}.$$
\end{proposition}
\begin{proof}
This results from an evaluation of  $\frac{1}{x} \text{Rev}\left(\frac{x(1-yx)}{(1+x)^2}\right)$.
\end{proof}
\begin{corollary}
We have
$$B(x,y)=\frac{1}{1-2x}c\left(\frac{x(x+y)}{(1-2x)^2}\right),$$ where
$$c(x)=\frac{1-\sqrt{1-4x}}{2x}$$ is the generating function of the Catalan numbers $C_n=\frac{1}{n+1}\binom{2n}{n}$ \seqnum{A000108} \cite{Stanley}.
\end{corollary}
\begin{corollary} The generating function of the row sums of the Borel triangle is given by
$$\frac{1-2x-\sqrt{1-8x}}{2(1+x)},$$ and the generating function of the diagonal sums of the Borel triangle is given by
$$\frac{1-2x-\sqrt{1-4x(1+x)}}{2x}.$$
\end{corollary}
\begin{proof}
We substitute $y=1$ for the row sums and $y=x$ for the diagonal sums.
\end{proof}
The row sums are sequence \seqnum{A062992}, which begins
$$1, 3, 13, 67, 381, 2307, 14589, 95235, 636925,\ldots.$$ It counts Dyck paths for which the up step, other than those starting at ground level, can come in two colors. The diagonal sums are the sequence \seqnum{A071356} that counts Motzkin paths where both the up steps and the level steps can have two colors.
\begin{definition} The sequence of polynomials $B_n(y)=\sum_{k=0}^n B_{n,k}y^k$ will be called the \emph{Borel polynomials}.\end{definition}
From the above corollary, the Borel polynomials are the image of the Catalan numbers by the bivariate Riordan array $\left(\frac{1}{1-2x}, \frac{x(x+y)}{(1-2x)^2}\right)$.
For instance, we have
\begin{scriptsize}
$$\left(
\begin{array}{ccccc}
 1 & 0 & 0 & 0 & 0 \\
 2 & y & 0 & 0 & 0 \\
 4 & 6 y+1 & y^2 & 0 & 0 \\
 8 & 6 (4 y+1) & 2 y (5 y+1) & y^3 & 0 \\
 16 & 8 (10 y+3) & 60 y^2+20 y+1 & y^2 (14 y+3) & y^4 \\
\end{array}
\right)\left(
\begin{array}{c}
 1 \\
 1 \\
 2 \\
 5  \\
 14  \\
\end{array}
\right)=\left(
\begin{array}{c}
 1 \\
 y+2 \\
 2 y^2+6 y+5 \\
 5 y^3+20 y^2+28 y+14 \\
 14 y^4+70 y^3+135 y^2+120 y+42 \\
\end{array}
\right).$$
\end{scriptsize}
\begin{corollary}
We have
$$B_{n,k} = \sum_{j=0}^n \binom{j}{k}\binom{n+k}{2j}2^{n+k-2j}C_j.$$
\end{corollary}
\begin{proof} The Riordan array $\left(\frac{1}{1-2x}, \frac{x(x+y)}{(1-2x)^2}\right)$ has general $n,k$-term given by
$$\sum_{k=0}^k \binom{k}{i}\binom{n+k-i}{n-k-i}2^{n-k-i}y^{k-i}.$$
Thus we have
$$B_n(y)=\sum_{j=0}^n \left(\sum_{k=0}^k \binom{k}{i}\binom{n+k-i}{n-k-i}2^{n-k-i}y^{k-i}\right)C_j.$$
The result now follows from $B_{n,k}=[y^k] B_n(y)$.
\end{proof}
This last expression allows us to generalize the Borel triangle.
\begin{definition} The $r$-Borel triangle $B_{n,k}^{(r)}$ is defined to be the triangle given by
$$B_{n,k}^{(r)}=\sum_{j=0}^n \binom{j}{k}\binom{n+k}{2j}r^{n+k-2j}C_j.$$
\end{definition}
For instance, the $1$-Borel triangle begins
$$\left(
\begin{array}{ccccccc}
 1 & 0 & 0 & 0 & 0 & 0 & 0 \\
 1 & 1 & 0 & 0 & 0 & 0 & 0 \\
 2 & 3 & 2 & 0 & 0 & 0 & 0 \\
 4 & 10 & 10 & 5 & 0 & 0 & 0 \\
 9 & 30 & 45 & 35 & 14 & 0 & 0 \\
 21 & 90 & 175 & 196 & 126 & 42 & 0 \\
 51 & 266 & 644 & 924 & 840 & 462 & 132 \\
\end{array}
\right).$$
The first column elements in this matrix are the Motzkin numbers \seqnum{A001006}.
\begin{proposition} The first column of the $r$-Borel triangle is the $r$-th binomial transform of the aerated Catalan numbers
$$1, 0, 1, 0, 2, 0, 5, 0, 14, 0, 42,\ldots.$$
\end{proposition}
\begin{proof} We have
$$B_{n,k}^{(r)} = \sum_{j=0}^n \binom{j}{k}\binom{n+k}{2j}r^{n+k-2j}C_j.$$
Thus the first column is given by
$$B_{n,0}^{(r)} = \sum_{j=0}^n \binom{n}{2j}r^{n-2j}C_j.$$
The statement follows from this.
\end{proof}

Note that the matrix $B_{n,k}^{(0)}$ begins
$$\left(
\begin{array}{ccccccccc}
 1 & 0 & 0 & 0 & 0 & 0 & 0 & 0 & 0 \\
 0 & 1 & 0 & 0 & 0 & 0 & 0 & 0 & 0 \\
 1 & 0 & 2 & 0 & 0 & 0 & 0 & 0 & 0 \\
 0 & 4 & 0 & 5 & 0 & 0 & 0 & 0 & 0 \\
 2 & 0 & 15 & 0 & 14 & 0 & 0 & 0 & 0 \\
 0 & 15 & 0 & 56 & 0 & 42 & 0 & 0 & 0 \\
 5 & 0 & 84 & 0 & 210 & 0 & 132 & 0 & 0 \\
 0 & 56 & 0 & 420 & 0 & 792 & 0 & 429 & 0 \\
 14 & 0 & 420 & 0 & 1980 & 0 & 3003 & 0 & 1430 \\
\end{array}
\right).$$
This is an aerated version of the triangle \seqnum{A085880} with general element $\binom{n}{k}C_n$. The aerated triangle has its generating function is given by
$$\frac{1-\sqrt{1-4x(x+y)}}{2(x+y)}.$$

\section{The Borel triangle as a moment coefficient array}
Riordan arrays of the form
$$\left(\frac{1+ \lambda x + \mu x^2}{1+ \alpha x + \beta x^2}, \frac{x}{1+ \alpha x + \beta x^2}\right)$$ are the coefficient arrays of families of generalized Chebyshev polynomials \cite{Classical}. The elements of the first column of the inverse matrix are then the moments corresponding to the family of orthogonal polynomials. We now exhibit the Borel polynomials as a family of moments.
\begin{proposition} The Borel polynomials $B_n(y)$ are the moments of the family of orthogonal polynomials defined by the Riordan array
$$\left(\frac{1+xy}{(1+(y+1)x)^2}, \frac{x}{(1+(y+1)x)^2}\right).$$
\end{proposition}
\begin{proof}
We must show that the elements of the first column of the inverse matrix are the polynomials $B_n(y)$.
To calculate the inverse matrix we must first find $\bar{f}=\text{Rev}\left( \frac{x}{(1+(y+1)x)^2}\right)$.
We find that $$\bar{f}=\frac{1-2x(y+1)-\sqrt{1-4x(y+1)}}{2x(y+1)^2}.$$
Next, with $g(x)=\frac{1+xy}{(1+(y+1)x)^2}$, we need to calculate $\frac{1}{g(\bar{f})}$.
We find that the first column elements of the inverse matrix have generating function given by
$$\frac{1}{g(\bar{f})}=\frac{1-2x-\sqrt{1-4x(y+1)}}{2x(x+y)}=B(x,y),$$ as required.
\end{proof}
\begin{corollary}
The Borel triangle is the coefficient array of the moments of the orthogonal polynomials defined by the three-term recurrence
$$P_n(x)=(x-2(y+1))P_{n-1}-(y+1)^2 P_{n-2}(x),$$ with
$P_0(x)=1,  P_1(x)=x-y-2$.
\end{corollary}
The matrix $\left(\frac{1+xy}{(1+(y+1)x)^2}, \frac{x}{(1+(y+1)x)^2}\right)$ begins
$$\left(
\begin{array}{cccc}
 1 & 0 & 0 & 0 \\
 -y-2 & 1 & 0 & 0 \\
 (y+1) (y+3) & -3 y-4 & 1 & 0 \\
 (-y-4) (y+1)^2 & 2 (y+1) (3 y+5) & -5 y-6 & 1 \\
\end{array}
\right),$$ and the corresponding orthogonal polynomials $P_n(y)$ begin
$$1, x - y - 2, x^2 - x(3·y + 4) + (y + 1)(y + 3), x^3 - x^2(5y + 6) + 2x(y + 1)(3y + 5) - (y + 4)(y + 1)^2,\ldots.$$
The production matrix of the inverse array begins
$$\left(
\begin{array}{ccccc}
 y+2 & 1 & 0 & 0 & 0 \\
 (y+1)^2 & 2 (y+1) & 1 & 0 & 0 \\
 0 & (y+1)^2 & 2 (y+1) & 1 & 0 \\
 0 & 0 & (y+1)^2 & 2 (y+1) & 1 \\
 0 & 0 & 0 & (y+1)^2 & 2 (y+1) \\
\end{array}
\right),$$ and we can show that this pattern continues. From this we get the following result.
\begin{proposition}
The generating function $B(x,y)$ of the Borel triangle has the following Jacobi continued fraction expression.
$$B(x,y)=
\cfrac{1}{1-(y+2)x-
\cfrac{(y+1)^2x^2}{1-2(y+1)x-
\cfrac{(y+1)^2x^2}{1-2(y+1)x-\cdots}}}.$$
\end{proposition}
This result allows us to find a matrix factorization of the Borel triangle in terms of the so-called Catalan triangle \seqnum{A009766}
$$C_{n,k}=\frac{n-k+1}{k+1} \binom{n+k}{k}[k \le n]$$ and the binomial matrix $\left(\binom{n}{k}\right)$.
\begin{proposition}
We have the factorization
$$ \left(B_{n,k}\right)=\left(C_{n,k}\right)\cdot \left(\binom{n}{k}\right).$$
\end{proposition}
\begin{proof}
The generating function
$$\cfrac{1}{1-(y+1)x-
\cfrac{y^2x^2}{1-2yx-
\cfrac{y^2x^2}{1-2yx-\cdots}}}$$ obtained by changing $y$ in the above generating function $B(x,y)$ to $y-1$ results in the generating function of the triangle $\left(B_{n,k}\right) \cdot \left(\binom{n}{k}\right)^{-1}$.
But the generating function
$$\cfrac{1}{1-(y+1)x-
\cfrac{y^2x^2}{1-2yx-
\cfrac{y^2x^2}{1-2yx-\cdots}}}$$ is that of the Catalan triangle.
\end{proof}

\section{The general case of $B_{n,k}^{(r,s)}$.}
We define $\left(B_{n,k}^{(r,s)}\right)$ to be the number triangle whose $(n,k)$-th element is given by
$$B_{n,k}^{(r,s)}=\sum_{j=0}^n \binom{j}{k}\binom{n+k}{2j}s^j r^{n+k-2j}C_j.$$
We define the $(r,s)$-Borel polynomials to be the polynomials $\sum_{k=0}^n B_{n,k}^{(r,s)}y^k$. We then have the following characterization of these polynomials as a moment sequence.
\begin{proposition} The $(r,s)$-Borel polynomials are the moments of the family of orthogonal polynomials whose coefficient array is given by the Riordan array
$$\left(\frac{1+sxy}{1+(2sy+r)x+s(1+ry+sy^2)y^2}, \frac{x}{1+(2sy+r)x+s(1+ry+sy^2)y^2}\right).$$
\end{proposition}
\begin{corollary} The bivariate generating function of the $(r,s)$-Borel triangle is given by the following expression.
$$\frac{2}{1-rx+\sqrt{1-2(r+2sy)x+(r^2-4s)x^2}},$$ or equivalently,
$$\frac{1-rx-\sqrt{1-2(r+2sy)x+(r^2-4s)x^2}}{2s(x+y)}.$$
\end{corollary}
\begin{corollary} The generating function $B^{(r,s)}(x,y)$ of the $(r,s)$-Borel triangle can be expressed as the following Jacobi continued fraction.
$$\cfrac{1}{1-(r+sy)x-
\cfrac{(s+rsy+s^2y^2)x^2}{1-(r+2sy)x-
\cfrac{(s+rsy+s^2y^2)x^2}{1-(r+2sy)x-\cdots}}}.$$
\end{corollary}

\begin{example} For $s=0, r=1$ we get the triangle with generating function
$$\frac{1-4x(x+y)}{2x(x+y)},$$ that begins
$$\left(
\begin{array}{ccccccc}
 1 & 0 & 0 & 0 & 0 & 0 & 0 \\
 0 & 1 & 0 & 0 & 0 & 0 & 0 \\
 1 & 0 & 2 & 0 & 0 & 0 & 0 \\
 0 & 4 & 0 & 5 & 0 & 0 & 0 \\
 2 & 0 & 15 & 0 & 14 & 0 & 0 \\
 0 & 15 & 0 & 56 & 0 & 42 & 0 \\
 5 & 0 & 84 & 0 & 210 & 0 & 132 \\
\end{array}
\right).$$
The row sums of this matrix begin
$$ 1, 1, 3, 9, 31, 113, 431, 1697, 6847, 28161,\ldots.$$
This is \seqnum{A052709}, which counts the number of walks from $(0,0)$ to $(n,0)$ using steps $(1,1), (1,-1)$ and $(0,-1)$.
\end{example}
\begin{example} The case $r=1, s=1$ gives us the triangle with generating function
$$\frac{1-x-\sqrt{1-2(2y+1)x-3x^2}}{2x(x+y)}$$ which begins
$$\left(
\begin{array}{ccccccc}
 1 & 0 & 0 & 0 & 0 & 0 & 0 \\
 1 & 1 & 0 & 0 & 0 & 0 & 0 \\
 2 & 3 & 2 & 0 & 0 & 0 & 0 \\
 4 & 10 & 10 & 5 & 0 & 0 & 0 \\
 9 & 30 & 45 & 35 & 14 & 0 & 0 \\
 21 & 90 & 175 & 196 & 126 & 42 & 0 \\
 51 & 266 & 644 & 924 & 840 & 462 & 132 \\
\end{array}
\right).$$
The first column is the Motzkin numbers. The row sums are the sequence \seqnum{A064641} which begin
$$1, 2, 7, 29, 133, 650, 3319, 17498, 94525,\ldots.$$
This sequence has the generating function
$$\frac{1-x-\sqrt{1-6x+3x^2}}{2x(1+x)}.$$ This sequence counts the number of paths from $(0,0)$ to $(n,n)$ not rising above $y=x$, using steps $(1,0), (0,1), (1,1)$ and $(2,1)$
.
The diagonal sums of this triangle have generating function given by
$$\frac{1-x-\sqrt{1-2x-7x^2}}{4x^2}.$$
This expands to give the sequence \seqnum{A025235}, which begins
$$	1, 1, 3, 7, 21, 61, 191, 603, 1961,\ldots.$$
This sequence counts the number of lattice paths in the first quadrant from $(0,0)$ to $(n,0)$ using only steps H$=(1,0)$, U$=(1,1)$ and D$=(1,-1)$, where the U steps come in two colors.
\end{example}
We recall that we have defined the Borel triangle as follows.
$$B_{n,k}=[x^n y^k] \frac{1}{x} \text{Rev}\left(\frac{x(1-yx)}{(1+x)^2}\right).$$
We now define a generalization using two parameters $r$ and $s$ as follows.
$$\tilde{B}_{n,k}^{(r,s)}=[x^n y^k] \frac{1}{x} \text{Rev}\left(\frac{x(1-yx)}{1+rx+sx^2}\right).$$
Using Lagrange inversion as before, we find the following.
\begin{proposition} We have
$$\tilde{B}_{n,k}^{(r,s)}=\frac{1}{n+1}\binom{n+k}{k}\sum_{j=0}^{n+1} \binom{n+1}{j}\binom{j}{n-k-j}s^{n-k-j}r^{2j+k-n}.$$
\end{proposition}
The term
$$\sum_{j=0}^{n+1} \binom{n+1}{j}\binom{j}{n-k-j}s^{n-k-j}r^{2j+k-n}$$ is of interest in its own right.
We have the following result.
\begin{proposition} The $(n,k)$-th term of the element of the derivative subgroup of the Riordan group specified by
$$\left(\left(\frac{x}{1+rx+sx^2}\right)', \frac{x}{1+rx+sx^2}\right)^{-1}$$ is given by
$$t_{n,k}=\sum_{j=0}^{n+1} \binom{n+1}{j}\binom{j}{n-k-j}s^{n-k-j}r^{2j+k-n}.$$
\end{proposition}

Using the expression $ \frac{1}{x} \text{Rev}\left(\frac{x(1-yx)}{1+rx+sx^2}\right)$ for the generating function of the triangle $\tilde{B}_{n,k}^{(r,s)}$, we have the following result.
\begin{proposition} The generating function of the generalized Borel triangle $\tilde{B}_{n,k}^{(r,s)}$ is given by the following expression.
$$\tilde{B}^{(r,s)}(x,y)=\frac{2}{1-rx+\sqrt{1-2(r+2y)x+(r^2-4s)x^2}}.$$
Equivalently, we have
$$\tilde{B}^{(r,s)}(x,y)=\frac{1-rx+\sqrt{1-2(r+2y)x+(r^2-4s)x^2}}{2x(sx+y)}.$$
\end{proposition}
\begin{corollary} The generating function of the first column of the generalized Borel triangle $\tilde{B}_{n,k}^{(r,s)}$ is given by
$$\frac{1}{x} \text{Rev} \left(\frac{x}{1+rx+sx^2}\right).$$
\end{corollary}
\begin{proof} The generating function of the first column of the generalized Borel triangle $\tilde{B}_{n,k}^{(r,s)}$ is obtained by setting $y=0$ in the generating function for the triangle. This gives
$$\frac{1-rx+\sqrt{1-2rx+(r^2-4s)x^2}}{2sx^2},$$
as required.
\end{proof}
We then have the following result which characterizes the generalized Borel polynomials $\tilde{B}_n^{(r,s)}(y)=\sum_{k=0}^n \tilde{B}_{n,k}^{(r,s)}y^k$.
\begin{proposition}
The generalized Borel polynomials $\tilde{B}_n^{(r,s)}(y)$ are the moments for the family of orthogonal polynomials whose coefficient array is given by the following Riordan array.
$$\left(\frac{1+xy}{1+(r+2y)x+(s+ry+y^2)x^2}, \frac{x}{1+(r+2y)x+(s+ry+y^2)x^2}\right).$$
\end{proposition}
\begin{corollary}
The generating function of the generalized Borel polynomials $\tilde{B}_n^{(r,s)}(y)$ can be expressed as the following Jacobi continued fraction.
$$\cfrac{1}{1-(r+y)x-
\cfrac{(s+ry+y^2)x^2}{1-(r+2y)x-
\cfrac{(s+ry+y^2)x^2}{1-(r+2y)x-\cdots}}}.$$
\end{corollary}
We have the following alternative expression for $\tilde{B}_{n,k}^{(r,s)}$.
$$\tilde{B}_{n,k}^{(r,s)}=\sum_{j=0}^n \binom{j}{k}\binom{n+k}{2j}s^{j-k} r^{n+k-2j}C_j.$$
\begin{example} When $r=2$ and $s=3$, we obtain the triangle $\tilde{B}_{n,k}^{(2,3)}$ that begins
$$\left(
\begin{array}{ccccccc}
 1 & 0 & 0 & 0 & 0 & 0 & 0 \\
 2 & 1 & 0 & 0 & 0 & 0 & 0 \\
 7 & 6 & 2 & 0 & 0 & 0 & 0 \\
 26 & 36 & 20 & 5 & 0 & 0 & 0 \\
 106 & 200 & 165 & 70 & 14 & 0 & 0 \\
 452 & 1095 & 1190 & 728 & 252 & 42 & 0 \\
 1999 & 5922 & 8036 & 6384 & 3150 & 924 & 132 \\
\end{array}
\right).$$
Here, the sequence
$$1, 2, 7, 26, 106, 452, 1999,\ldots$$ is \seqnum{A122871}, whose generating function is
$$\frac{1}{x} \text{Rev}\left(\frac{x}{1+2x+3x^2}\right).$$
Dividing the $(n,k)$-th element of this matrix by $\frac{1}{n+1} \binom{n+k}{k}$ yields the triangle that begins $$\left(
\begin{array}{ccccccc}
 1 & 0 & 0 & 0 & 0 & 0 & 0 \\
 4 & 1 & 0 & 0 & 0 & 0 & 0 \\
 21 & 6 & 1 & 0 & 0 & 0 & 0 \\
 104 & 36 & 8 & 1 & 0 & 0 & 0 \\
 530 & 200 & 55 & 10 & 1 & 0 & 0 \\
 2712 & 1095 & 340 & 78 & 12 & 1 & 0 \\
 13993 & 5922 & 2009 & 532 & 105 & 14 & 1 \\
\end{array}
\right).$$
This is the matrix
$$\left(\left(\frac{x}{1+2x+3x^2}\right)', \frac{x}{1+2x+3x^2}\right)^{-1}.$$
\end{example}
\begin{example} \textbf{The Borel triangle as a Hadamard product} The above shows that we can express the Borel triangle as the following Hadamard product.
$$B= \left(\frac{1}{n+1} \binom{n+k}{k}\right) \circ \left(\left(\frac{x}{(1+x)^2}\right)', \frac{x}{(1+x)^2}\right)^{-1},$$
where $\circ$ signifies the Hadamard operation of entry-wise multiplication.
\end{example}
\section{The generating function $xB^{(r)}(x,y)$ as a zero}
We have the following proposition showing that the generating functions $xB^{(r)}(x,y)$ are zeros of a quadratic.
\begin{proposition} The generating function $xB^{(r)}(x,y)$ is a zero of the quadratic in $z$ given by
$$ z^2-\frac{1-rx}{x+y}z+ \frac{x}{x+y}.$$
\end{proposition}
In particular, $xB^{(2)}(x,y)=xB(x,y)$ is one of the zeros of
$$ z^2-\frac{1-2x}{x+y}z+ \frac{x}{x+y}.$$
\begin{proof} Solving the equation
$$z^2-\frac{1-rx}{x+y}z+ \frac{x}{x+y}=0$$
gives the result.
\end{proof}
In like manner, we have the following result concerning the generating function $C(x,y)$ of the Catalan triangle.
\begin{proposition} The generating function $xC(x,y)$ is a zero of the quadratic polynomial
$$ z^2-\frac{1-2x}{x+y-1}z+ \frac{x}{x+y-1}.$$
\end{proposition}
With regard to the generalized Borel triangle $\tilde{B}_{n,k}^{(r,s)}$, we have the following result.
\begin{proposition} The generating function $\tilde{B}^{(r,s)}(x,y)$ is a zero of the quadratic in $z$ given by
$$z^2-\frac{1-rx}{x(sx+y)}z+\frac{x^2+y}{x(sx+y)}.$$
\end{proposition}

\section{Reverting the triangles}
In this section, we briefly look at the reversions of the Borel and Catalan triangles. By the reversion of the Borel triangle, for instance, we mean the matrix whose generating function is given by
$$\frac{1}{x} \text{Rev}_x xB(x,y).$$
Similarly, by the reversion of the Catalan triangle, we mean the matrix whose generating function is given by
$$\frac{1}{x} \text{Rev}_x xC(x,y).$$
We have the following results.
\begin{proposition}
The reversion of the Borel triangle is the matrix that begins
$$\left(
\begin{array}{cccccccc}
 1 & 0 & 0 & 0 & 0 & 0 & 0 & 0 \\
 -2 & -1 & 0 & 0 & 0 & 0 & 0 & 0 \\
 3 & -2 & 0 & 0 & 0 & 0 & 0 & 0 \\
 -4 & -3 & 0 & 0 & 0 & 0 & 0 & 0 \\
 5 & -4 & 0 & 0 & 0 & 0 & 0 & 0 \\
 -6 & -5 & 0 & 0 & 0 & 0 & 0 & 0 \\
 7 & -6 & 0 & 0 & 0 & 0 & 0 & 0 \\
 -8 & -7 & 0 & 0 & 0 & 0 & 0 & 0 \\
\end{array}
\right),$$ with generating function
$$\frac{1-yx}{(1+x)^2}.$$
\end{proposition}
We note that by multiplying this matrix by the binomial matrix (on the left), we get the matrix
$$\left(
\begin{array}{ccccccc}
 1 & 0 & 0 & 0 & 0 & 0 & 0 \\
 -1 & -1 & 0 & 0 & 0 & 0 & 0 \\
 0 & 0 & 0 & 0 & 0 & 0 & 0 \\
 0 & 0 & 0 & 0 & 0 & 0 & 0 \\
 0 & 0 & 0 & 0 & 0 & 0 & 0 \\
 0 & 0 & 0 & 0 & 0 & 0 & 0 \\
 0 & 0 & 0 & 0 & 0 & 0 & 0 \\
\end{array}
\right).$$
Thus there is a sequence of invertible transforms from the Borel triangle to this matrix.

\begin{proposition}
The reversion of the Catalan triangle is the matrix that begins
$$\left(
\begin{array}{cccccccc}
 1 & 0 & 0 & 0 & 0 & 0 & 0 & 0 \\
 -1 & -1 & 0 & 0 & 0 & 0 & 0 & 0 \\
 1 & -2 & 0 & 0 & 0 & 0 & 0 & 0 \\
 -1 & -3 & 0 & 0 & 0 & 0 & 0 & 0 \\
 1 & -4 & 0 & 0 & 0 & 0 & 0 & 0 \\
-1 & -5 & 0 & 0 & 0 & 0 & 0 & 0 \\
 1 & -6 & 0 & 0 & 0 & 0 & 0 & 0 \\
-1 & -7 & 0 & 0 & 0 & 0 & 0 & 0 \\
\end{array}
\right),$$
with generating function
$$\frac{1+x-yx}{(1+x)^2}.$$
\end{proposition}
Multiplying this matrix on the left by the binomial matrix yields the diagonal matrix
$$\left(
\begin{array}{ccccccc}
 1 & 0 & 0 & 0 & 0 & 0 & 0 \\
 0 & -1 & 0 & 0 & 0 & 0 & 0 \\
 0 & 0 & 0 & 0 & 0 & 0 & 0 \\
 0 & 0 & 0 & 0 & 0 & 0 & 0 \\
 0 & 0 & 0 & 0 & 0 & 0 & 0 \\
 0 & 0 & 0 & 0 & 0 & 0 & 0 \\
 0 & 0 & 0 & 0 & 0 & 0 & 0 \\
\end{array}
\right).$$ Note that this matrix has its reversion given by the diagonal matrix which begins
$$\left(
\begin{array}{ccccccc}
 1 & 0 & 0 & 0 & 0 & 0 & 0 \\
 0 & 1 & 0 & 0 & 0 & 0 & 0 \\
 0 & 0 & 2 & 0 & 0 & 0 & 0 \\
 0 & 0 & 0 & 5 & 0 & 0 & 0 \\
 0 & 0 & 0 & 0 & 14 & 0 & 0 \\
 0 & 0 & 0 & 0 & 0 & 42 & 0 \\
 0 & 0 & 0 & 0 & 0 & 0 & 132 \\
\end{array}
\right).$$
We note also that the reversion of the ``reversal'' of the Catalan triangle
$$\left(
\begin{array}{cccccc}
 1 & 0 & 0 & 0 & 0 & 0 \\
 1 & 1 & 0 & 0 & 0 & 0 \\
 2 & 2 & 1 & 0 & 0 & 0 \\
 5 & 5 & 3 & 1 & 0 & 0 \\
 14 & 14 & 9 & 4 & 1 & 0 \\
 42 & 42 & 28 & 14 & 5 & 1 \\
\end{array}
\right)$$ begins
$$\left(
\begin{array}{cccccccc}
 1 & 0 & 0 & 0 & 0 & 0 & 0 & 0 \\
 -1 & -1 & 0 & 0 & 0 & 0 & 0 & 0 \\
 0 & 2 & 1 & 0 & 0 & 0 & 0 & 0 \\
 0 & 0 & -3 & -1 & 0 & 0 & 0 & 0 \\
 0 & 0 & 0 & 4 & 1 & 0 & 0 & 0 \\
 0 & 0 & 0 & 0 & -5 & -1 & 0 & 0 \\
 0 & 0 & 0 & 0 & 0 & 6 & 1 & 0 \\
 0 & 0 & 0 & 0 & 0 & 0 & -7 & -1 \\
\end{array}
\right).$$
This is the exponential Riordan array $[1-x,-x]$. Its inverse is the exponential Riordan array $\left[\frac{1}{1+x}, -x\right]$. It is interesting to note that the reversion of this latter matrix is the ordinary Riordan array $(g(x), xg(x))$ where $g(x)$ is the ordinary generating function of the factorial numbers $n!$.

As a point of comparison, the reversion of the binomial matrix $\left(\binom{n}{k}\right)$ is the order two matrix $\left((-1)^n \binom{n}{k}\right)$ that begins
$$\left(
\begin{array}{cccccccc}
 1 & 0 & 0 & 0 & 0 & 0 & 0 & 0 \\
 -1 & -1 & 0 & 0 & 0 & 0 & 0 & 0 \\
 1 & 2 & 1 & 0 & 0 & 0 & 0 & 0 \\
 -1 & -3 & -3 & -1 & 0 & 0 & 0 & 0 \\
 1 & 4 & 6 & 4 & 1 & 0 & 0 & 0 \\
 -1 & -5 & -10 & -10 & -5 & -1 & 0 & 0 \\
 1 & 6 & 15 & 20 & 15 & 6 & 1 & 0 \\
 -1 & -7 & -21 & -35 & -35 & -21 & -7 & -1 \\
\end{array}
\right).$$
\section{Fuss-Borel triangles}
We have seen that we can describe the Borel triangle in terms of the Hadamard product of two matrices, one of which is a Riordan array. Thus we have
$$B= \left(\frac{1}{n+1} \binom{n+k}{k}\right) \circ \left(\left(\frac{x}{(1+x)^2}\right)', \frac{x}{(1+x)^2}\right)^{-1}.$$
We now define a one parameter family of triangles, which we call \emph{Fuss-Borel} triangles, as follows.
$$B(r)=\left(\frac{1}{n+1} \binom{(r-1)(n+1)+k-1}{k}\right) \circ \left(\left(\frac{x}{(1+x)^r}\right)', \frac{x}{(1+x)^r}\right)^{-1}.$$
We then have
$$B(r)_{n,k}=\frac{1}{n+1}\binom{(r-1)(n+1)+k-1}{k}\binom{r(n+1)}{n-k}.$$
For $r=0\ldots 4$, we get the triangles that begin
$$\left(
\begin{array}{cccccc}
 1 & 0 & 0 & 0 & 0 & 0 \\
 0 & -1 & 0 & 0 & 0 & 0 \\
 0 & 0 & 1 & 0 & 0 & 0 \\
 0 & 0 & 0 & -1 & 0 & 0 \\
 0 & 0 & 0 & 0 & 1 & 0 \\
 0 & 0 & 0 & 0 & 0 & -1 \\
\end{array}
\right),\left(
\begin{array}{cccccc}
 1 & 0 & 0 & 0 & 0 & 0 \\
 1 & 0 & 0 & 0 & 0 & 0 \\
 1 & 0 & 0 & 0 & 0 & 0 \\
 1 & 0 & 0 & 0 & 0 & 0 \\
 1 & 0 & 0 & 0 & 0 & 0 \\
 1 & 0 & 0 & 0 & 0 & 0 \\
\end{array}
\right),$$
$$\left(
\begin{array}{cccccc}
 1 & 0 & 0 & 0 & 0 & 0 \\
 2 & 1 & 0 & 0 & 0 & 0 \\
 5 & 6 & 2 & 0 & 0 & 0 \\
 14 & 28 & 20 & 5 & 0 & 0 \\
 42 & 120 & 135 & 70 & 14 & 0 \\
 132 & 495 & 770 & 616 & 252 & 42 \\
\end{array}
\right),\left(
\begin{array}{cccccc}
 1 & 0 & 0 & 0 & 0 & 0 \\
 3 & 2 & 0 & 0 & 0 & 0 \\
 12 & 18 & 7 & 0 & 0 & 0 \\
 55 & 132 & 108 & 30 & 0 & 0 \\
 273 & 910 & 1155 & 660 & 143 & 0 \\
 1428 & 6120 & 10608 & 9282 & 4095 & 728 \\
\end{array}
\right),$$ and
$$\left(
\begin{array}{cccccc}
 1 & 0 & 0 & 0 & 0 & 0 \\
 4 & 3 & 0 & 0 & 0 & 0 \\
 22 & 36 & 15 & 0 & 0 & 0 \\
 140 & 360 & 312 & 91 & 0 & 0 \\
 969 & 3420 & 4560 & 2720 & 612 & 0 \\
 7084 & 31878 & 57684 & 52440 & 23940 & 4389 \\
\end{array}
\right).$$
Multiplying these on the left by the inverse binomial matrix, we obtain a family of \emph{Fuss-Catalan} triangles $C(r)_{n,k}$ which for $r=0\ldots 4$ begin
$$\left(
\begin{array}{cccccc}
 1 & 0 & 0 & 0 & 0 & 0 \\
 1 & -1 & 0 & 0 & 0 & 0 \\
 1 & -2 & 1 & 0 & 0 & 0 \\
 1 & -3 & 3 & -1 & 0 & 0 \\
 1 & -4 & 6 & -4 & 1 & 0 \\
 1 & -5 & 10 & -10 & 5 & -1 \\
\end{array}
\right),\left(
\begin{array}{cccccc}
 1 & 0 & 0 & 0 & 0 & 0 \\
 1 & 0 & 0 & 0 & 0 & 0 \\
 1 & 0 & 0 & 0 & 0 & 0 \\
 1 & 0 & 0 & 0 & 0 & 0 \\
 1 & 0 & 0 & 0 & 0 & 0 \\
 1 & 0 & 0 & 0 & 0 & 0 \\
\end{array}
\right),$$
$$\left(
\begin{array}{cccccc}
 1 & 0 & 0 & 0 & 0 & 0 \\
 1 & 1 & 0 & 0 & 0 & 0 \\
 1 & 2 & 2 & 0 & 0 & 0 \\
 1 & 3 & 5 & 5 & 0 & 0 \\
 1 & 4 & 9 & 14 & 14 & 0 \\
 1 & 5 & 14 & 28 & 42 & 42 \\
\end{array}
\right),\left(
\begin{array}{cccccc}
 1 & 0 & 0 & 0 & 0 & 0 \\
 1 & 2 & 0 & 0 & 0 & 0 \\
 1 & 4 & 7 & 0 & 0 & 0 \\
 1 & 6 & 18 & 30 & 0 & 0 \\
 1 & 8 & 33 & 88 & 143 & 0 \\
 1 & 10 & 52 & 182 & 455 & 728 \\
\end{array}
\right),$$ and
$$\left(
\begin{array}{cccccc}
 1 & 0 & 0 & 0 & 0 & 0 \\
 1 & 3 & 0 & 0 & 0 & 0 \\
 1 & 6 & 15 & 0 & 0 & 0 \\
 1 & 9 & 39 & 91 & 0 & 0 \\
 1 & 12 & 72 & 272 & 612 & 0 \\
 1 & 15 & 114 & 570 & 1995 & 4389 \\
\end{array}
\right).$$
The first column elements of the Fuss-Borel triangles are given by $\frac{1}{(r-1)n+r}\binom{r(n+1)}{n+1}$. For $r=0\ldots 5$, we get the sequences
$$
\begin{array}{ccccccccc}
 1, & 0, & 0, & 0, & 0, & 0, & \ldots& \seqnum{A000007}& 0^n\\
 1, & 1, & 1, & 1, & 1, & 1, & \ldots& \seqnum{A000012}& 1\\
 1, & 2, & 5, & 14, & 42, & 132, & \ldots&\seqnum{A000108}(n+1)& \frac{1}{n+2}\binom{2n+2}{n+1}\\
 1, & 3, & 12, & 55, & 273, & 1428, & \ldots& \seqnum{A001764}(n+1)& \frac{1}{2n+3} \binom{3n+3}{n+1}\\
 1, & 4, & 22, & 140, & 969, & 7084, & \ldots&\seqnum{A002293}(n+1)&\frac{1}{3n+4}\binom{4n+4}{n+1} \\
 1, & 5, & 35, & 285, & 2530, & 23751, & \ldots&\seqnum{A002294}(n+1)& \frac{1}{4n+5} \binom{5n+5}{n+1}.\\
\end{array}
$$
The diagonal elements are given by $\frac{1}{n+1}\binom{r(n+1)-2}{n}$. For $r=0 \ldots 5$, we get the sequences
$$
\begin{array}{ccccccccc}
 1, & -1, & 1, & -1, & 1, & -1, & \ldots &\seqnum{A033999} & (-1)^n\\
 1, & 0, & 0, & 0, & 0, & 0, & \ldots& \seqnum{A000007} & 0^n \\
 1, & 1, & 2, & 5, & 14, & 42, & \ldots & \seqnum{A000108} & \frac{1}{n+1}\binom{2n}{n}=\frac{1}{2n+1}\binom{2n+1}{n+1}\\
 1, & 2, & 7, & 30, & 143, & 728, & \ldots& \seqnum{A00613}&  \frac{1}{n+1}\binom{3n+1}{n}=\frac{1}{3n+2}\binom{3n+2}{n+2} \\
 1, & 3, & 15, & 91, & 612, & 4389, & \ldots &\seqnum{A006632}& \frac{1}{n+1}\binom{4n+2}{n}=\frac{1}{4n+3}\binom{4n+3}{n+1} \\
 1, & 4, & 26, & 204, & 1771, & 16380,  & \ldots &\seqnum{A118971}& \frac{1}{n+1}\binom{5n+3}{n}=\frac{1}{5n+4}\binom{5n+4}{n+1}\\
\end{array}
$$
Note that for $r \ge 2$, these sequences are positive definite \cite{PD}. 

We can characterize the reversal of the Fuss-Catalan triangles as follows.
\begin{proposition} The reversal of the Fuss-Catalan triangle for $r$ is the Riordan array
$$\left((1-x)^{r-1}, x(1-x)^{r-1}\right)^{-1}.$$
\end{proposition}
\begin{proof}
We have
$$\left((1-x)^{r-1}, x(1-x)^{r-1}\right)^{-1}=\left(\frac{1}{x}\text{Rev}\left(x(1-x)^{r-1}\right), \text{Rev}\left(x(1-x)^{r-1}\right)\right).$$
Thus the $(n,k)$-th element of this array is given by
\begin{align*}
a_{n,k}&=[x^n] \frac{1}{x}\text{Rev}\left(x(1-x)^{r-1}\right)\left(\text{Rev}\left(x(1-x)^{r-1}\right)\right)^k\\
&=[x^{n+1}] \left(\text{Rev}\left(x(1-x)^{r-1}\right)\right)^{k+1}\\
&=\frac{1}{n+1}[x^n](k+1)x^k\left(\frac{1}{(1-x)^{r-1}}\right)^{n+1}\\
&=\frac{k+1}{n+1}[x^{n-k}]\left(\frac{1}{1-x}\right)^{(r-1)(n+1)}\\
&=\frac{k+1}{n+1}[x^{n-k}]\sum_j^{\infty} \binom{-(r-1)(n+1)}{j}(-1)^j x^j\\
&=\frac{k+1}{n+1}[x^{n-k}]\sum_j \binom{(r-1)(n+1)+j-1}{j}x^j\\
&=\frac{k+1}{n+1} \binom{(r-1)(n+1)+n-k-1}{n-k}.\end{align*}
Now the reversal of
$$\frac{k+1}{n+1} \binom{(r-1)(n+1)+n-k-1}{n-k}$$ is
$$\frac{n-k+1}{n+1} \binom{(r-1)(n+1)+k-1}{k}$$
as required.
\end{proof}
\begin{corollary}
We have
$$B(r)_{n,k}=\sum_{j=0}^n \frac{n-j+1}{n+1} \binom{(r-1)(n+1)+j-1}{j}\binom{j}{k}.$$
\end{corollary}
The Fuss-Catalan triangle $C(2)_{n,k}$ is the Catalan triangle. The array $C(3)_{n,k}$, which begins
$$\left(
\begin{array}{cccccc}
 1 & 0 & 0 & 0 & 0 & 0 \\
 1 & 2 & 0 & 0 & 0 & 0 \\
 1 & 4 & 7 & 0 & 0 & 0 \\
 1 & 6 & 18 & 30 & 0 & 0 \\
 1 & 8 & 33 & 88 & 143 & 0 \\
 1 & 10 & 52 & 182 & 455 & 728 \\
\end{array}
\right),$$ and its reversal are \seqnum{A071948} and \seqnum{A092276}, respectively.

A natural generalization of the Fuss-Borel triangles is to introduce parameters by replacing the simple polynomial $(1+x)^r$ by an arbitrary polynomials $\sum_{i=0}^r t_i x^i$ to get a family of generalized Fuss-Borel triangles $B(r)(t_0,t_1,\ldots, t_r)$.
\begin{example}
We have
$$B(3)(1,1,1,1)=\left(\frac{1}{n+1}\binom{2n+k+1}{k}\right)\circ \left(\left(\frac{x}{1+x+x^2+x^3}\right)', \frac{x}{1+x+x^2+x^3}\right)^{-1}.$$
This triangle begins
$$\left(
\begin{array}{ccccccc}
 1 & 0 & 0 & 0 & 0 & 0 & 0 \\
 1 & 2 & 0 & 0 & 0 & 0 & 0 \\
 2 & 6 & 7 & 0 & 0 & 0 & 0 \\
 5 & 20 & 36 & 30 & 0 & 0 & 0 \\
 13 & 70 & 165 & 220 & 143 & 0 & 0 \\
 36 & 240 & 728 & 1274 & 1365 & 728 & 0 \\
 104 & 826 & 3045 & 6720 & 9520 & 8568 & 3876 \\
\end{array}
\right).$$
The first column numbers
$$	1, 1, 2, 5, 13, 36, 104, 309, 939, 2905, 9118,\ldots$$ give the sequence \seqnum{A036765}. This sequence counts the number of Dyck paths that avoid $UUUU$, as well as counting the number of non-crossing partitions of $[n]$ in which all blocks are of size less than or equal to $3$. The sequence that begins
$$1, 2, 7, 30, 143, 728, 3876, 21318, 120175,$$ is the sequence \seqnum{A006013} or $\frac{1}{n+1}\binom{3n+1}{n}$, which enumerates pairs of ternary trees.
\end{example}
\begin{example} The triangle $B(4)(1,1,1,1,1)$ begins
$$\left(
\begin{array}{ccccccc}
 1 & 0 & 0 & 0 & 0 & 0 & 0 \\
 1 & 3 & 0 & 0 & 0 & 0 & 0 \\
 2 & 9 & 15 & 0 & 0 & 0 & 0 \\
 5 & 30 & 78 & 91 & 0 & 0 & 0 \\
 14 & 105 & 360 & 680 & 612 & 0 & 0 \\
 41 & 378 & 1596 & 3990 & 5985 & 4389 & 0 \\
 125 & 1365 & 6930 & 21252 & 42504 & 53130 & 32890 \\
\end{array}
\right).$$
The first column is \seqnum{A036766}, which counts the number of Dyck paths of length $n$ the avoid $UUUUU$. It also counts the number of ordered rooted trees with $n$ non-root nodes with all out-degrees less than or equal to $4$. The elements on the diagonal
$$1, 3, 15, 91, 612, 4389, 32890,\ldots$$ give the sequence \seqnum{A006632}, or $\frac{1}{n+1}\binom{4n+2}{n}=\frac{1}{4n+3}\binom{4n+3}{n+1}$.
\end{example}

\section{A note on the Catalan triangle}
The Catalan triangle
$$C_{n,k}=\frac{n-k+1}{n+1}\binom{n+k}{k} [k<=n]$$ has generating function
$$\frac{1-2x-\sqrt{1-4xy}}{2x(x+y-1)}.$$
We can express this generating function as
$$\frac{1}{1-x-y}-\frac{1-\sqrt{1-4xy}}{2x(1-x-y)}=\frac{1}{1-x-y}-\frac{yc(xy)}{1-x-y}.$$
The generating function $\frac{1}{1-x-y}$ is that of the symmetric matrix  $\left(\binom{n+k}{k}\right)$, which begins
$$\left(
\begin{array}{cccccc}
 1 & 1 & 1 & 1 & 1 & 1 \\
 1 & 2 & 3 & 4 & 5 & 6 \\
 1 & 3 & 6 & 10 & 15 & 21 \\
 1 & 4 & 10 & 20 & 35 & 56 \\
 1 & 5 & 15 & 35 & 70 & 126 \\
 1 & 6 & 21 & 56 & 126 & 252 \\
\end{array}
\right).$$
The generating function $\frac{c(xy)}{1-x-y}$ is that of the symmetric matrix
$$\left(
\begin{array}{ccccccc}
 1 & 1 & 1 & 1 & 1 & 1 & 1 \\
 1 & 3 & 4 & 5 & 6 & 7 & 8 \\
 1 & 4 & 10 & 15 & 21 & 28 & 36 \\
 1 & 5 & 15 & 35 & 56 & 84 & 120 \\
 1 & 6 & 21 & 56 & 126 & 210 & 330 \\
 1 & 7 & 28 & 84 & 210 & 462 & 792 \\
 1 & 8 & 36 & 120 & 330 & 792 & 1716 \\
\end{array}
\right).$$
This has general term
$$\sum_{i=0}^k \binom{n+k-2i}{n-i}C_i=\sum_{i=0}^k \binom{i-k-1}{n-i}(-1)^{n-i}C_i.$$
We then have
$$C_{n,k}=\binom{n+k}{k}-\sum_{i=0}^{k-1} \binom{n+k-2i-1}{n-i}C_i$$ or equivalently
$$C_{n,k}=\binom{n+k}{k}-\sum_{i=0}^{k-1} \binom{i-k}{n-i}(-1)^{n-i}C_i.$$
For instance, we have
$$\left(
\begin{array}{cccccc}
 1 & 0 & 0 & 0 & 0 & 0 \\
 1 & 1 & 0 & 0 & 0 & 0 \\
 1 & 2 & 2 & 0 & 0 & 0 \\
 1 & 3 & 5 & 5 & 0 & 0 \\
 1 & 4 & 9 & 14 & 14 & 0 \\
 1 & 5 & 14 & 28 & 42 & 42 \\
\end{array}
\right)=$$
$$\left(
\begin{array}{cccccc}
 1 & 1 & 1 & 1 & 1 & 1 \\
 1 & 2 & 3 & 4 & 5 & 6 \\
 1 & 3 & 6 & 10 & 15 & 21 \\
 1 & 4 & 10 & 20 & 35 & 56 \\
 1 & 5 & 15 & 35 & 70 & 126 \\
 1 & 6 & 21 & 56 & 126 & 252 \\
\end{array}
\right)-\left(
\begin{array}{cccccc}
 0 & 1 & 1 & 1 & 1 & 1 \\
 0 & 1 & 3 & 4 & 5 & 6 \\
 0 & 1 & 4 & 10 & 15 & 21 \\
 0 & 1 & 5 & 15 & 35 & 56 \\
 0 & 1 & 6 & 21 & 56 & 126 \\
 0 & 1 & 7 & 28 & 84 & 210 \\
\end{array}
\right).$$
(Note that care must be exercised in working with the above binomial expressions where negative parameters are encountered, as different software packages may give different results. The above results which coincide with the expansions of the generating functions have been calculated using Derive.)

\section{Conclusion} We have used Riordan array techniques to analyse and characterize the Borel triangle. We have further used these techniques to generalize the Borel triangle in a number of directions. Because of the closeness of the Borel triangle to the so-called Catalan triangle, the Catalan numbers and their generalizations are of fundamental importance. The reversion of bivariate generating functions along one of the variables has proved useful in further elucidating structural elements of these matrices.

\bigskip
\hrule

\noindent 2010 {\it Mathematics Subject Classification}: Primary
15B36; Secondary 11B83, 11C20, 33C45
\noindent \emph{Keywords:} Borel triangle, Borel polynomial, Riordan group, Riordan array, Catalan triangle, triangle reversion, orthogonal polynomial, moment

\bigskip
\hrule
\bigskip
\noindent (Concerned with sequences
\seqnum{A000007},
\seqnum{A000012},
\seqnum{A000045},
\seqnum{A000108},
\seqnum{A001006},
\seqnum{A001045},
\seqnum{A001764},
\seqnum{A002293},
\seqnum{A002294},
\seqnum{A006013},
\seqnum{A006632},
\seqnum{A007318},
\seqnum{A009766},
\seqnum{A025235},
\seqnum{A036765},
\seqnum{A036766},
\seqnum{A033999},
\seqnum{A052709},
\seqnum{A062992}
\seqnum{A064641},
\seqnum{A071356},
\seqnum{A071948},
\seqnum{A085880},
\seqnum{A092276},
\seqnum{A118971},
\seqnum{A122871}, and
\seqnum{A234950}.)

\end{document}